\documentclass{article}[11pt]
\usepackage{amsmath,amsthm,amsfonts}
\usepackage{amsfonts,amsmath,latexsym}
\usepackage{graphicx}
\usepackage{epsfig}
\usepackage{color}
\newtheorem{theorem}{Theorem}

\newtheorem{lemma}[theorem]{Lemma}
\newtheorem{remark}[theorem]{Remark}

\newcommand{\Ba}{\partial_t^{1-\alpha}}
\newcommand{\R}{\mathbb{R}}
\newcommand{\fInt}{\mathcal{I}}
\newcommand{\iprod}[1]{\left\langle#1\right\rangle}

\newcommand{\W}{\mathcal W}
\newcommand{\matM}{{[\bf{M}]}}
\newcommand{\matG}{{[{\bf G}^n]}}
\newcommand{\matS}{{[{\bf S}^n]}}

\newcommand{\mvec}[1]{{{#1}}}
\newcommand{\vnabla}{\nabla}
\newcommand{\vF}{{\mvec{F}}}

\newcommand{\vpsi}{{{\psi}}}
\begin{document}

\title{A second-order accurate numerical scheme for a time-fractional Fokker-Planck equation}
%

\author{K. Mustapha\thanks{Department of Mathematics and Statistics,
King Fahd University of Petroleum and Minerals,
Dhahran, 31261, Saudi Arabia}, O. M. Knio\thanks{Computer, Electrical, Mathermatical Sciences and Engineering Division, King Abdullah University of Science and Technology, Thuwal 23955, Saudi Arabia}, O. P. Le Ma\^{\i}tre\thanks{CNRS, LIMSI, Universit\'e Paris-Scalay, Campus Universitaire - BP 133, F-91403 Orsay, France}}

\maketitle

\begin{abstract}
{A time-stepping $L1$  scheme for solving a  time fractional Fokker-Planck equation of order $\alpha \in (0, 1)$,  with a general driving force, is investigated.   A stability bound for the semi-discrete solution is obtained for $\alpha\in(1/2,1)$ {via a  novel and concise approach.} Our stability estimate is $\alpha$-robust in the sense that it remains valid in the limiting case where $\alpha$ approaches $1$ (when the model reduces to the classical Fokker-Planck equation), a limit that presents practical importance. Concerning the error analysis, we obtain an optimal second-order accurate estimate for  $\alpha\in(1/2,1)$. A time-graded mesh is used to compensate for the singular behavior of the continuous solution near the origin. The $L1$ scheme is associated with a standard spatial Galerkin finite element discretization to numerically support our theoretical contributions. We employ the resulting fully-discrete computable numerical scheme to perform some numerical tests. These tests suggest that the imposed time-graded meshes assumption could be further relaxed, and we observe second-order accuracy even for the case $\alpha\in(0,1/2]$, that is, outside the range covered by the theory.}{Fractional Fokker-Planck, $L1$ approximations, finite elements, stability and error analysis, graded meshes}
\end{abstract}

\section{Introduction}
We consider the following time-fractional Fokker-Planck equation (Angstmann et al. (2015))
\begin{equation}\label{eq: ibvp}
\partial_t u(x,t)
	+\vnabla \cdot \Big( \partial_t^{1-\alpha}\kappa_\alpha \vnabla  u+ \vF \partial_t^{1-\alpha} u\Big) (x,t)	=f(x,t),
	\quad\mbox{for } x\in\Omega \mbox{ and } 0<t<T,
\end{equation}
where $\Omega$ is a convex polyhedral domain in~$\R^d$ ($d\ge1$).
In~\eqref{eq: ibvp}, the diffusivity coefficient $\kappa_\alpha$, function of $x$, is assumed uniformly bounded with a lower bound $\kappa_{\min} >0$. The driving force $\vF:=(F_1,F_2,\ldots,F_d)$ is a vector function of $x$ and $t$, while $f$ is a source term. The functions $\kappa_\alpha$, $\vF$,  and $f$ are assumed to be sufficiently regular functions on their respective domains.   
Further,  $\partial_t=\partial/\partial t$ is the classical time partial  derivative,   and the time-fractional derivative $\partial^{1-\alpha} /\partial t$  is taken in the Riemann--Liouville sense, that is
\[
\partial_t^{1-\alpha} v=\partial_t\fInt^\alpha v, ~~
\fInt^\alpha v(t)
	= \int_0^t\omega_\alpha(t-s)v(s)\,ds \mbox{ with }
	\omega_\alpha(t)=\frac{t^{\alpha-1}}{\Gamma(\alpha)},~~ {\rm with}~~0<\alpha<1.
\]
The problem is completed with the initial condition $u(x,0)=u_0(x)$, and homogeneous Dirichlet boundary condition
$u(x,t)=0$ for $x\in\partial\Omega$ and $0<t<T.$  The well-posedness and regularity properties of problem~\eqref{eq: ibvp} and more general models were recently investigated in (Le et al. (2019), McLean et al. (2019) and (2020)). 
\begin{remark}
 Assuming $u$ satisfies the regularity assumption in  \eqref{eq: regularity2}  below, the time-stepping stability property and error bound established in this paper for the homogeneous Dirichlet boundary condition remain valid in the case of a zero-flux boundary condition,
\begin{equation}\label{eq: Neuman bc}
\partial_t^{1-\alpha}\left( \kappa_\alpha\frac{\partial u}{\partial n} \right)
	-(\mvec{F}\cdot\mvec{n})\,\partial_t^{1-\alpha}u=0
	\quad \mbox{ for } x\in\partial\Omega\mbox{ and } 0<t<T,
\end{equation}
where $\mvec{n}$ denotes the outward unit normal to~$\Omega$.  In this case, one should substitute the Sobolev space $H^1_0(\Omega)$ with $H^1(\Omega)$ throughout the paper. Here, the rate of change of the total mass~$\int_\Omega  u(\cdot,t)$ is equal to  $\int_\Omega  f(\cdot,t)$ and thus, the total mass is  conserved in the absence of the external source $f.$  
\end{remark}

{
The particular case of $\vF$ independent of~$t$ and $f \equiv 0$  allows to reformulate the  fractional Fokker-Planck equation by applying $\mathcal{I}^{1-\alpha}$ on both side of the governing equation~\eqref{eq: ibvp} to obtain 
\begin{equation}\label{eq: Caputo}
{}^{\mathrm{C}}\partial_t^\alpha u-\vnabla\cdot\bigl(\kappa\vnabla u-\vF u\bigr)
    =0,
\end{equation}
where ${}^\mathrm{C}\partial_t^\alpha u=\mathcal{I}^{1-\alpha}\partial_tu$ is the Caputo fractional derivative of order~$\alpha$.  Numerous authors have studied the numerical solution of~\eqref{eq: Caputo}, mostly for a 1D spatial domain~$\Omega=(0,L)$ and with $f\equiv0$. For example,  Deng (2007) considered the method of lines, Jiang and  Xu (2019) proposed a finite volume method, Saadatmandi et al. (2012) investigated a collocation method based on time-shifted Legendre polynomials and sinc functions in space,  Yang et  al. (2018) proposed a spectral collocation method, and Duong and Jin (2020) a Wasserstein gradient flow formulation. For the case of  $\vF = \mvec{0}$ in~\eqref{eq: Caputo}, various time-stepping methods were investigated, including the popular $L1$ schemes, see for instance~(Jin et al. (2018), Karaa and Pani (2020), Kopteva (2019), Liao et al. (2019), Stynes et al. (2017), Wang et al. (2018), Yan et al. (2018)). For more details, see the recent survey by Stynes (2021) and related references therein. 
}

{
Space-time-dependent driving forces $\vF$ make much more challenging the analysis of numerical schemes for problem~\eqref{eq: ibvp}, especially the time-discretization stability and error. 
For the spatial discretization error, a Galerkin finite element method was previously studied in (Le et al. (2016), Huang et al. (2020)) and the analyses assumed sufficiently regular data and a fractional exponent $\alpha \in (1/2,1)$. An error analysis for non-smooth data and $0<\alpha<1$ was presented in (Le et al. (2018)), and more recently in (McLean and Mustapha (2021)) where a uniform ($\alpha$-robust) stability bound was proved. Concerning the time-discretization of~\eqref{eq: ibvp}, a semi-discrete backward Euler time-stepping method was proposed in (Le et al. (2016)). Therein, some new ideas were introduced to show an $O(\tau^\alpha)$-rate for sufficiently regular data $u_0$, $\vF$, $f$, uniform time-mesh with a step-size $\tau$, and a fractional exponent $\alpha \in (1/2,1)$. Similar convergence results were proved later under analogous assumptions for a slightly modified scheme (Huang et al. (2020)). We are not aware of other works on the time discretization of~\eqref{eq: ibvp}. 
}

{
The present work is the first paper to develop and analyze a second-order  accurate time-stepping method for solving~\eqref{eq: ibvp} using $L1$ approximations. Since  the continuous solution $u$ of~\eqref{eq: ibvp} has singularity near $t=0$, a time-graded mesh is employed to improve the convergence rate and achieve the optimal order of accuracy. Stability and optimal convergence analysis are carried out for $1/2<\alpha < 1$, which is the range of practical interest for the fractional exponent (diffusion and sub-diffusion). The results extend to $0<\alpha<1$ in the case of zero initial data ($u_0=0$). A similar $L1$ time-stepping scheme was recently analyzed in (Mustapha (2020)) in the case of zero driving force ($\vF=\mvec{0}$). Unfortunately, the error analysis therein  can not be extended to nonzero time-space driving forces $\vF$ (precisely, the proof of the main error result (Lemma 3.1, Mustapha (2020)) is problematic). Furthermore, even for the case of zero $\vF$, the stability proof of the numerical scheme is also missing. 
To derive the stability bound and prove the optimal convergence rate, the present work involves  original technical contributions. We reformulate the numerical scheme in the convenient compact form~\eqref{eq: scheme W} leading to the weak formulation in~\eqref{eq: compact}. Applying the operator $\fInt^\alpha$ to~\eqref{eq: compact} we obtain a new weak formulation. Then, proceeding with carefully selected test functions in these weak formulations, we derive several new results using some properties of the fractional integral (see Lemmas~\ref{lem: pointwise bound}, \ref{lem: I nu I mu} and~\ref{lem: y(t)}). In addition, we prove in Lemma~\ref{lem: Gronwall} a discrete version of a weakly singular Gronwall's inequality for a graded time-mesh. The achieved stability results are exploited to carry the error analysis for both uniform and graded time meshes.}
  
The organization of the paper is as follows. In Section~\ref{sec: scheme}, we define our semi-discrete $L1$ time-stepping scheme. We also state and show some preliminary results that will be used later in our stability and error analyses.  
In Section~\ref{sec: stability}, we propose a novel approach to prove a stability bound of the numerical solution when $1/2<\alpha < 1$, see Theorem~\ref{thm: stability}. As mentioned above, this stability bound remains valid for $0<\alpha < 1$ in the case of zero initial data $u_0=0$.  Section~\ref{sec: fractional interplation} focuses on estimating the errors resulting from applying the fractional derivative to the difference between a function and its piecewise linear polynomial interpolant  over uniform and non-uniform time meshes, where Lemmas 3.2 and 3.4 of (Mustapha (2020)) are used. These error estimates are used later in the convergence analysis of the $L1$ scheme.   Section~\ref{sec: error analysis} presents our error analysis. We assumed that the continuous solution $u$ of problem~\eqref{eq: ibvp} satisfies the following regularity property:   
for $0<\alpha<1,$ 
\begin{equation}\label{eq: regularity2}
\|
\partial_t^{1-\alpha} u(t)\|+t\|\partial_t^{2-\alpha} u(t)\|\le M\,t^{\alpha-1},\quad 
 \|u''(t)\|_1+t\|u'''(t)\|_1
  \le M t^{\sigma+\alpha/2-2},
\end{equation} 
 for some positive constants  $M$ and $\sigma$ with  $0<\sigma\le  \alpha$, which is  the likely situation for reasonable regular data. Above, the prime ($'$) denotes the time partial derivative and $\|\cdot\|_\ell$ is the norm on the usual Sobolev space $H^\ell(\Omega)$ which reduces to the $L_2(\Omega)$-norm, simply denoted $\|\cdot\|$, when $\ell=0$.  As an example,  when $t^j \|f^{(j)}(t)\|_1\le c\, t^{\frac{3\alpha}{4} -1}$ for $0\le j\le 3$  ($f^{(j)}$ denotes the $j^{th}$ time partial derivative of $f$), and $u_0\in H^1_0(\Omega)\cap H^{2+q}(\Omega)$ for some $0<q<0.5$, the second assumption in~\eqref{eq: regularity2} is true for $\sigma=q\,\frac{\alpha}{2}$. However, it is sufficient to assume $u_0\in L^2(\Omega)$ to ensure that the first assumption holds; for more details on the regularity results, see~(McLean (2010), McLean et al. (2020))  for zero or nonzero $\vF$, respectively. 

It is worth mentioning that for non-smooth data $u_0$ and $f$, $\sigma$ is expected to be $\le 0$. Investigating these situations is beyond the scope of the current work and will be a topic of future research. The main result in Theorem~\ref{thm: time convergence} concerns  the sub-optimal $O(\tau^{\sigma+\alpha/2})$-rate of convergence over a uniform time mesh, and  optimal $O(\tau^2)$-rate of convergence over time-graded meshes with a mesh exponent $\gamma > 2/(\sigma+\alpha/2)$ (see~\eqref{eq: time mesh} for the definition of the time-graded mesh). In both cases, $1/2< \alpha< 1$  is assumed.    

In Section~\ref{sec: numerical results}, we illustrate the theoretical convergence results numerically on a sample of test problems.  For this purpose, we combine the $L1$ time-stepping scheme with a standard continuous (linear) Galerkin finite element method for the spatial discretization, then defining a fully-discrete scheme. For such an approach, one can check that the stability estimate in Theorem~\ref{thm: stability} remains valid with $u_{0h}$ in place $u_0$. The fully-discrete $L1$ finite element scheme is briefly introduced in Section~\ref{sec: fully22}.  We present several numerical results for $0<\alpha<1$, that is, not only in the range $(1/2,1)$ as theoretically  necessary.  The numerical results suggest $O(\tau^{\min\{\gamma(\sigma+\alpha),2\}})$-rates of convergence for $\gamma \ge 1.$ Hence, optimal convergence rates can be achieved even for time mesh exponent $\gamma \ge  2/(\sigma+\alpha)$. This finding means that the  time-graded mesh  constrain  can be relaxed numerically,   and the achieved convergence rate on the uniform mesh can be further improved.

\section{$L1$-time stepping scheme and preliminary results}\label{sec: scheme}

This section is devoted to our semi-discrete time-stepping $L1$ numerical scheme for solving the model problem~\eqref{eq: ibvp}. We use a time-graded mesh with nodes $t_i$ defined as follows. Let $\gamma\ge 1$ and denote $N$ the number of time-intervals. We set   
\begin{equation} \label{eq: time mesh} 
  \tau=T^{1/\gamma}/N, \quad t_i=(i\,\tau)^\gamma,\quad \mbox{for } 0\le i\le N.
\end{equation}
Such meshes are used in different contexts, including Volterra integral equations and super- and sub-diffusion models (see for instance~ (Brunner et al. (1999), Chandler and Graham (1988), McLean and Mustapha (2007), Mustapha (2015), Stynes et al. (2017)), to compensate for singular behaviour in derivatives. For $1\le n\le N$, we denote $\tau_n=t_n-t_{n-1}$ the length of the $n$-th subinterval~$I_n=(t_{n-1},t_n)$. The time-graded mesh has the following properties~ (McLean and Mustapha (2007)): for $n\ge 2$, 
\begin{equation}\label{eq: mesh property 1}
t_n\le 2^\gamma t_{n-1},\quad \gamma \tau t_{n-1}^{1-1/\gamma}\le \tau_n\le \gamma \tau t_n^{1-1/\gamma},\quad \tau_n-\tau_{n-1} \le c_\gamma \tau^2 \,t_n^{1-2/\gamma},
\end{equation}
where $c_\gamma$ is a non-negative constant depending on $\gamma$ only, which is zero for $\gamma=1$. 

To define our time-stepping numerical scheme, we integrate problem~\eqref{eq: ibvp} over the time interval $I_n$, 
\begin{equation}\label{eq: ibvp int}
u(t_n)-u(t_{n-1}) -\int_{t_{n-1}}^{t_n}[\vnabla \cdot(\partial_t^{1-\alpha}(\kappa_\alpha  \vnabla u)-\vF\partial_t^{1-\alpha} u)](t)\,dt= \int_{t_{n-1}}^{t_n} f(t)\,dt\,.
\end{equation}

We define our $L1$ approximate solution  $U(t) \approx u(t)$ to be continuous and piecewise linear polynomial in time over each closed subinterval $[t_{n-1},t_n]$, that is 
\[
	U(t)=\frac{(t-t_{n-1})}{\tau_n} U(t_n)+\frac{(t_n-t)}{\tau_n} U(t_{n-1})\,.
\] 

 Motivated by \eqref{eq: ibvp int} and noticing that $U(t_n)-U(t_{n-1})= U'(t)$ for $t\in I_n$, the approximate solution satisfies
\begin{equation} \label{fully}
  U'(t)- \frac{1}{\tau_n}\int_{t_{n-1}}^{t_n}[\vnabla \cdot(\partial_s^{1-\alpha}(\kappa_\alpha  \vnabla U)-\vF^{n-\frac{1}{2}}\partial_s^{1-\alpha} U)]\,ds= \frac{1}{\tau_n}\int_{t_{n-1}}^{t_n} f\,ds,~~{\rm for}~~t \in I_n, 
\end{equation}
for $ 1\le n\le N$, with $U(0)=u_0$ and $\vF^{n-\frac{1}{2}}=\frac{\vF(t_{n-1})+\vF(t_n)}{2}.$  For the case of non-smooth initial data $u_0$, the scheme above can be modified by replacing $U(t)$ with a (time) constant function  $U(t_1)$ for $t \in I_1.$  Noting that, in the limiting case $\alpha \rightarrow 1$, problem~\eqref{eq: ibvp} reduces to the classical Fokker-Planck equation with external source and time-space dependent driving force; the time-stepping scheme then corresponds to the second-order accurate Crank-Nicolson time-stepping scheme.

In the rest of the section we present four lemmas which provide inequalities for subsequent stability and error analyses.
The first lemma (Lemma 2.3, McLean et al. (2019)) will later enable us to establish pointwise estimates for certain terms, where $\iprod{\cdot,\cdot}$ denotes the $L_2$-inner product on the spatial domain $\Omega$.
\begin{lemma}\label{lem: pointwise bound}
  Let $0\le \mu <\nu \le 1$. If the function $\chi: [0,T] \to L_2(\Omega)$ is continuous with $\chi(0)=0$, and if its restriction to $(0,T]$ is piecewise differentiable with $\chi'(t)=Ct^{-\mu}$ for $0<t\le T$,  then
  \[
    \|\chi(t)\|^2\le 2\omega_{2-\nu}(t)\int_0^{t}\iprod{\fInt^\nu\chi'(t),\chi'(t)}\,dt\,.
  \]
\end{lemma}

\begin{lemma}\label{lem: I nu I mu}
If $0\le\nu\le1$, then for $t>0$~and $\chi\in L_2\bigl((0,t),L_2(\Omega)\bigr)$,
\[
\int_0^t\|(\mathcal{I}^\nu\chi)(s)\|^2\,ds\le 2t^{2\nu}\int_0^t
    \|\chi(s)\|^2\,ds.
\]
\end{lemma}
\begin{proof}
See Lemma 3.1 in (Le et al. (2018)).
\end{proof}

\begin{lemma}\label{lem: y(t)}
If~ $0<\nu \le 1$, then for $\chi \in L_2\bigr((0,t),L_2(\Omega)\bigr)$,
\[
  \int_0^t\Big(\fInt^\nu(\|\chi(s)\|)\Big)^2\,ds\le\omega_{\nu+1}(t) \int_0^t\omega_\nu(t-s)\int_0^s\|\chi(q)\|^2\,dq\,ds.
\]
\end{lemma}
\begin{proof} The proof is identical to the proof of Lemma~2.3 in (Le et al. (2016)).
\end{proof}

The next lemma extends the discrete Gronwall's inequality in (Dixon and McKee (1986) and Ye et al. (2007)) to the case of time-graded meshes of the form~\eqref{eq: time mesh}. 
\begin{lemma}\label{lem: Gronwall}
  Let $0 <\nu \le  1$,  and let $t_i$ be defined as in \eqref{eq: time mesh}. Assume that $(A_n)^N_{n=0}$  is a nonnegative and nondecreasing sequence and that $B \ge  0$. The nonegative sequence $(y_n)^N_{n=0}$ satisfies 
  \[
    y_n \le A_n + B \sum^{n-1}_{j=0}y_j\int_{t_{j-1}}^{t_j} \omega_\nu(t_n -t)\,dt,\quad \mbox{for } 0\le  n \le N.
  \]
 Then, for $0\le n\le N$, it comes $y_n \le  E_\nu(B \gamma^\nu\,T^{\nu(\gamma-1)} t_n^{\nu/\gamma}) A_n$, where $E_\nu(\cdot)=E_{\nu,1}(\cdot)$ is the Mittag-Leffler function (see  the definition in Section~\ref{sec: numres}).
\end{lemma}
\begin{proof}
  By using  the second  mesh property in \eqref{eq: mesh property 1}, the inequality $n^\gamma -j^\gamma \ge  \gamma (n-j) j^{\gamma-1}$, and the equality $\tau^\nu=T^{\nu/\gamma}N^{-\nu},$  we have for $0\le j\le n-1$
\begin{align*}
\int_{t_{j-1}}^{t_j} (t_n -t)^{\nu-1}\,dt &\le \tau_j (t_n -t_j)^{\nu-1}\le \gamma \tau t_j^{1-1/\gamma}((n\tau)^\gamma -(j\tau)^\gamma)^{\nu-1}\\
&= \gamma   \tau^{\gamma\nu} j^{\gamma-1} (n^\gamma -j^\gamma)^{\nu-1}
 \le \gamma^\nu   \tau^{\gamma\nu} j^{\nu(\gamma-1)} (n -j)^{\nu-1}\\
&=\gamma^\nu   \tau^\nu t_j^{\nu(\gamma-1)} (n -j)^{\nu-1} = \gamma^\nu T^{\nu/\gamma} t_j^{\nu(\gamma-1)}  N^{-\nu}   (n -j)^{\nu-1}.
\end{align*}
Hence, 
\[
  0 \le  y_n \le A_n + B\frac{\gamma^\nu T^{\nu/\gamma+\nu(\gamma-1)}}{\Gamma(\nu)}\, \Big(\frac{1}{N}\Big)^\nu   \sum^{n-1}_{j=0} \frac{y_j}{(n -j)^{1-\nu}},\quad \mbox{for }0\le  n \le N.
\]
Therefore, an application of the Gronwall's inequality in~ (Theorem~6.1, Dixon and McKee (1986)) yields the desired result immediately. 
\end{proof}


\section{Stability analysis}\label{sec: stability}
This section is devoted to show the stability of the $L1$ time-stepping scheme~\eqref{fully}. 
Throughout the rest of the paper, $C$ is a generic constant which may depend on the parameters $M$, $\sigma$, $T$, $\Omega$,   and $\gamma$, but is independent of $\tau$ and $h$, where the latter is the maximum diameter of the spatial mesh element. 

Let $W(t)=U(t)-u_0$; because $W(0)=0$ we have $\partial_t^{1-\alpha} W(t)=\fInt^\alpha  W'(t)$ and our scheme in ~\eqref{fully} can be rewritten in the compact form    
\begin{equation}\label{eq: scheme W}
W'(t)- \vnabla\cdot\Big(\kappa_\alpha   \vnabla \phi(t)  -\big(\overline \vF\phi\big)(t) \Big)
 = \overline  f(t) + \vnabla\cdot \vpsi(t),\quad \mbox{for } t\in I_n,
 \end{equation}
with the piecewise constant functions  (in the time variable)
\[ 
  \phi(t\in I_n) = \phi^n= \frac{1}{\tau_n}\int_{t_{n-1}}^{t_n}\fInt^\alpha  W'(s)\,ds,\quad 
 \overline f(t\in I_n)=\overline f^n= \frac{1}{\tau_n}\int_{t_{n-1}}^{t_n} f(s)\,ds,
\]
 and the piecewise constant vector functions  (in the time variable)
\[
\overline \vF(t\in I_n)=\vF^{n-\frac{1}{2}},\quad 
  \vpsi(t\in I_n)= \vpsi^n= \frac{1}{\tau_n}\int_{t_{n-1}}^{t_n} \Big( \kappa_\alpha \vnabla u_0   +\overline \vF(s) u_0\Big) \omega_\alpha(s) \,ds.
\]
For later use, we take the $L_2$-inner product of~\eqref{eq: scheme W} with a test function $v \in H^1_0(\Omega)$; applying the first Green identity over $\Omega$, we obtain for $1\le n\le N$,  
\begin{equation}\label{eq: compact} 
 \iprod{W'(t),v}+ \iprod{(\kappa_\alpha \vnabla \phi-\overline \vF \phi)(t),\vnabla v}
 = \iprod{\bar f(t),v}-\iprod{\vpsi (t),\vnabla v},\quad \mbox{for } t\in I_n.
 \end{equation}

A preliminary stability estimate is derived in the next lemma. In the proof, we use the notations: $c_0=\max_{1\le j\le N} \|\vF^{j-\frac{1}{2}}\|_{L_\infty(\Omega)}$, $c_1=\frac{c_0^2}{\kappa_{\min}}+\frac{1}{2},$ $c_{2,n}=\frac{1}{2c_1}+4t_n^{2\alpha}$, $c_{3,n}=\frac{1}{c_1\kappa_{\min}}+\frac{32c_1}{\kappa_{\min}}t_n^{2\alpha},$ and $c_4=\frac{32c_1^2}{\Gamma({\alpha+1})}$. We assume that $8\omega_{\alpha+1}(t_n)\omega_{\alpha+1}(\tau_n)\Big(\frac{2c_0^2}{\kappa_{\min}}+1\Big)^2\le 1$, so that,  $\omega_{\alpha+1}(t_n) \tau_n^\alpha\le \frac{1}{c_4}$ for $1\le n\le N.$ This is not a restrictive  assumption as the inequality  can always be satisfied by using small enough time-steps ($N$ large enough). 

\begin{lemma}\label{lem: stab 1} 
 For $0<\alpha < 1$ and for $1\le n\le N,$ we have 
\[ \int_0^{t_n}\iprod{W',\fInt^\alpha W'}\,dt
 \le   C\sum_{j=1}^n \frac{1}{\tau_j}\Big(
 \Big\|\int_{t_{j-1}}^{t_j} f \,dt\Big\|^2
 +\Big\|\int_{t_{n-1}}^{t_n} \Big( \kappa_\alpha \vnabla u_0   +\overline \vF^n u_0\Big) \omega_\alpha(t) \,dt\Big\|^2\Big)\,.\]
 \end{lemma}
 \begin{proof} 
 Apply $\fInt^\alpha$ to both sides of~\eqref{eq: compact}, then choose  $v=\phi(t)$ and integrate in time over the interval 
 $I_n,$  
\[\int_{t_{n-1}}^{t_n} \iprod{ \fInt^\alpha W', \phi}\,dt + \int_{t_{n-1}}^{t_n} \iprod{\kappa_\alpha \fInt^\alpha \vnabla  \phi
-\fInt^\alpha(\overline \vF \phi),\vnabla \phi}\,dt
 = \int_{t_{n-1}}^{t_n} [\iprod{\fInt^\alpha \bar f,\phi}-\iprod{\fInt^\alpha \vpsi ,\vnabla \phi}]\,dt.
\]
And consequently,  applications of Cauchy-Schwarz and Young's inequalities yield
\begin{multline*}
  \int_{t_{n-1}}^{t_n} \|\phi\|^2\,dt + \int_{t_{n-1}}^{t_n} \iprod{ \kappa_\alpha  \fInt^\alpha \vnabla \phi,\vnabla \phi}\,dt
 \le \\ \int_{t_{n-1}}^{t_n} \| \vnabla \phi\|\Big(c_0\fInt^\alpha(\|\phi\|)+\|\fInt^\alpha \vpsi \|\Big)\,dt+\frac{1}{2}\int_{t_{n-1}}^{t_n} \Big( \|\fInt^\alpha \bar f\|^2+\|\phi\|^2\Big)\,dt\,.
 \end{multline*}
Summing over $n$ and using that $\int_0^{t_n} \iprod{ \kappa_\alpha  \fInt^\alpha \vnabla \phi,\vnabla \phi}\,dt\ge 0$, we observe 
\begin{equation}\label{eq: phi 1}
\int_0^{t_n} \|\phi\|^2\,dt 
   \le \epsilon  \int_0^{t_n} \|  \vnabla \phi\|^2\,dt + \frac{2}{\epsilon}  
 \int_0^{t_n} \Big(c_0^2\big(\fInt^\alpha(\|\phi\|)\big)^2+\|\fInt^\alpha \vpsi \|^2\Big)\,dt+
 \int_0^{t_n} \|\fInt^\alpha \bar f\|^2\,dt,
\end{equation}
for $\epsilon >0$. 
Now, choosing $v=\tau_n\phi^n$ in ~\eqref{eq: compact}, we notice that 
\[ 
    \int_{t_{n-1}}^{t_n}\iprod{W',\fInt^\alpha W'}\,dt+ \tau_n\|\sqrt{\kappa_\alpha} \vnabla\phi^n\|^2
      = \tau_n\iprod{\vF^{n-\frac{1}{2}} \phi^n-\vpsi ^n,\vnabla \phi^n}+\tau_n\iprod{\bar f^n,\phi^n}.
 \]  
Since the right-hand side is bounded by 
\[ 
  \tau_n \frac{\kappa_{\min}}{2} \| \vnabla\phi^n\|^2+\frac{\tau_n}{\kappa_{\min}}\big(c_0^2\|\phi^n\|^2+ \|\vpsi ^n\|^2\big) 
 + \frac{\tau_n}{2}\big(\|\bar f^n\|^2+\|\phi^n\|^2\big),\]
  and since $\kappa_\alpha \ge \kappa_{\min},$ we have
\[ 
  \int_{t_{n-1}}^{t_n}\iprod{W',\fInt^\alpha W'}\,dt +\tau_n\frac{\kappa_{\min}}{2}\| \vnabla\phi^n\|^2\le c_1\tau_n\|\phi^n\|^2 
 + \frac{\tau_n}{2}\|\bar f^n\|^2+\frac{\tau_n}{\kappa_{\min}} \|\vpsi ^n\|^2.
 \]
 Writing this  inequality in the integral form, then summing over $n$, we reach  
\begin{equation}\label{eq: stab}
  \int_0^{t_n}\iprod{W',\fInt^\alpha W'}\,dt+ \frac{\kappa_{\min}}{2} \int_0^{t_n}\| \vnabla\phi\|^2\,dt 
  \le c_1\int_0^{t_n} \|\phi\|^2\,dt 
  + \int_0^{t_n} \Big(\frac{1}{2}\|\bar f\|^2+\frac{1}{\kappa_{\min}}\|\vpsi \|^2\Big)\,dt,
 \end{equation}
for $1\le n\le N.$ However, the first term on the left-hand side is non-negative, so   
\[  \int_0^{t_n}\| \vnabla\phi\|^2\,dt \le \frac{2 c_1}{\kappa_{\min}}\int_0^{t_n} \|\phi\|^2\,dt 
   + \int_0^{t_n} \Big(\frac{1}{\kappa_{\min}}\|\bar f\|^2+\frac{2}{\kappa_{\min}^2}\|\vpsi \|^2\Big)\,dt.
\]
Inserting this result in~\eqref{eq: phi 1}, choosing $\epsilon=\frac{\kappa_{\min}}{4c_1}$ and simplifying, it comes
\begin{multline*}
\int_0^{t_n} \|\phi\|^2\,dt 
   \le   \int_0^{t_n} \Big(\frac{1}{2c_1}\|\bar f\|^2+\frac{1}{c_1\kappa_{\min}}\|\vpsi \|^2+\frac{16c_1}{\kappa_{\min}} \|\fInt^\alpha \vpsi \|^2+2\|\fInt^\alpha \bar f\|^2\Big)\,dt\\+   16c_1^2  
 \int_0^{t_n} \big(\fInt^\alpha(\|\phi\|)\big)^2\,dt.
 \end{multline*} 
An application of Lemma~\ref{lem: I nu I mu} with $\nu=\alpha$ gives 
\[
  \int_0^{t_n} \|\fInt^\alpha \bar f\|^2\,dt
  \le 2 t_n^{2\alpha} \int_0^{t_n}\| \bar f\|^2\,dt \quad \mbox{and} \quad\int_0^{t_n}\|\fInt^\alpha \vpsi \|^2\,dt
 \le 2 t_n^{2\alpha} \int_0^{t_n}\| \vpsi \|^2\,dt.
\]
Therefore, applying Lemma~\ref{lem: y(t)} we reach 
\[
\int_0^{t_n} \|\phi\|^2\,dt 
   \le   \int_0^{t_n} \Big(c_{2,n}\|\bar f\|^2+ c_{3,n}\| \vpsi \|^2\Big)\,dt\\+   \frac{c_4}{2}t_n^\alpha 
  \int_0^{t_n}\omega_\alpha(t_n-t) \int_0^t \|\phi\|^2\,dt.
\] 
Thus, for $1\le n\le N,$ 
\begin{multline*}
  (1-\frac{c_4}{2}\omega_{\alpha+1}(t_n) \tau_n^\alpha ) \int_0^{t_n} \|\phi\|^2\,dt   \le 
 \int_0^{t_n}\big( c_{2,n} \| \bar f\|^2+  c_{3,n}\| \vpsi \|^2 \big)\,dt\cr
 +\frac{c_4}{2} t_n^\alpha  \sum_{j=1}^{n-1} \int_{t_{j-1}}^{t_j}\omega_\alpha(t_n-t)\,dt  \int_0^{t_j}  \|\phi\|^2\,dt,
\end{multline*}
and hence, using $0\le c_4 \omega_{\alpha+1}(t_n)\tau_n^\alpha  \le  1$, we notice that  
\[ 
  \int_0^{t_n} \|\phi\|^2\,dt   \le 
 2\int_0^{t_n}\big( c_{2,n} \| \bar f\|^2+c_{3,n} \| \vpsi \|^2 \big)\,dt+ c_4 t_n^\alpha   \sum_{j=1}^{n-1} \int_{t_{j-1}}^{t_j}\omega_\alpha(t_n-t)\,dt  \int_0^{t_j}  \|\phi\|^2\,dt.
\] 
Applying the discrete Gronwall's inequality in Lemma~\ref{lem: Gronwall} yields  
\[
  \int_0^{t_n} \|\phi\|^2\,dt   \le 2 E_\alpha(c_4  \gamma^\nu\,T^{\alpha(\gamma-1)} t_n^{\alpha/\gamma+\alpha}) 
   \int_0^{t_n}\big(c_{2,n}\| \bar f\|^2+c_{3,n}\| \vpsi \|^2 \big)\,dt,\quad{\rm 
for}~~ 1\le n\le N.\] 
Thanks to the estimate~\eqref{eq: stab}, we have
\[ 
  \int_0^{t_n}\iprod{W',\fInt^\alpha W'}\,dt \le C\int_0^{t_n} \big(\|\bar f\|^2+\|\vpsi \|^2\big)\,dt.
\]
Recalling the definitions of the functions $\bar f$ and $\vpsi$ completes the proof. 
\end{proof}

We are now ready to show the main stability result of our numerical scheme.   Our stability estimate  remains valid as  $\alpha$ approaches $1$  where  problem \eqref{eq: ibvp} reduces to the classical Fokker-Planck equation.
\begin{theorem}\label{thm: stability}  
  The solution $U$ of the L1 time-stepping scheme~\eqref{fully} satisfies the following stability properties: 
  for $1\le n\le N,$  
  \[
    \|U(t_n)\|^2 \le C \|u_0\|_1^2 + Ct_n^{1-\alpha}\int_0^{t_n} \|f\|^2 \,dt\quad \mbox{for } \frac{1}{2}<\alpha< 1.
  \]
  In the case of zero initial data $(u_0=0)$, we have 
  \[
    \|U(t_n)\|^2 \le Ct_n^{1-\alpha}\int_0^{t_n} \|f\|^2 \,dt\quad{\rm for}~~ 0<\alpha< 1.
  \] 
\end{theorem}
\begin{proof}
Applying Lemma~\ref{lem: pointwise bound} and using  Lemma~\ref{lem: stab 1}, we notice that  
\begin{align*}
    \|W(t_n)\|^2 &\le  C t_n^{1-\alpha}  \int_0^{t_n}\iprod{W',\fInt^\alpha W'}\,dt
  \cr
  &\le   Ct_n^{1-\alpha} \sum_{j=1}^n \frac{1}{\tau_j}\Big[ \Big\|\int_{t_{j-1}}^{t_j} f \,dt\Big\|^2 +\Big(\int_{t_{j-1}}^{t_j} \omega_\alpha(t) \,dt\Big)^2 \|u_0\|_1^2\Big]
  \cr
  &\le   Ct_n^{1-\alpha} \int_0^{t_n} \|f\|^2 \,dt +Ct_n^{1-\alpha} \|u_0\|_1^2 \int_0^{t_n} \omega^2_\alpha(t) \,dt 
  \cr
  &\le   Ct_n^{1-\alpha} \int_0^{t_n} \|f\|^2 \,dt+\frac{Ct_n^\alpha}{2\alpha-1}\|u_0\|_1^2,
\end{align*}
where the assumption $\frac{1}{2}<\alpha< 1$ is used in the last inequality. Since $W(t_n)=U(t_n)-u_0$, $\|U(t_n)\|^2=\|W(t_n)+u_0\|^2 \le 2\|W(t_n)\|^2 + 2\|u_0\|^2$. Therefore, the above bound completes the proof.   
\end{proof} 

\section{Interpolation estimates}\label{sec: fractional interplation}
 
Throughout this section, we deal with a purely time dependent functions. 
We denote $\check g$ the piecewise linear polynomial that interpolates a given function $g$ at the time nodes $t_j$, that is, 
\begin{equation}
\label{eq: g check}
\check g(t) = \frac{t_j-t}{\tau_j}g(t_{j-1})+\frac{t-t_{j-1}}{\tau_j}g(t_j) \quad \mbox{for } t_{j-1}\le t\le t_j \mbox{ and }1\le j\le N.
\end{equation}  
The main interpolation error results of the section are Theorems \ref{thm: IE finalU} and \ref{thm: IE final} which  play a crucial role in the forthcoming error analysis  of  Section~\ref{sec: error}. The error estimates are derived for functions $g$ being singular near the origin. More precisely, we assume the following behavior as $t\rightarrow 0$ 
\begin{equation}
\label{eq: g reg}
|g''(t)|\le c t^{\nu-2}~~{\rm and}~~ |g'''(t)|\le c t^{\nu-3},\quad \mbox{for some } 0<\nu<1,
\end{equation} 
and for some positive constant $c.$ This form of singularity is suggested by the presence of the weakly singular kernel in our model  problem \eqref{eq: ibvp}.

For the case of a uniform time mesh,  we show in Theorem \ref{thm: IE finalU} that 
$$\sum_{j=1}^n \frac{1}{\tau}\Big|\int_{t_{j-1}}^{t_j} \Ba (g-\check g)(t)\,dt\Big|^2  \le C\tau^{2\nu}, ~~{\rm for}~~1/2<\alpha<1.$$  However, when $g''$ is continuous on the interval $[0,T],$ which is not the case here, by following the steps in Theorem \ref{thm: IE finalU}, we have  $O(\tau^4)$ convergence rate.    To maintain this order of accuracy when $g$ satisfies the regularity assumption in \eqref{eq: g reg} only and for $0<\alpha<1$,   we employ a time-graded mesh of the form~\eqref{eq: time mesh} setting $\gamma>2/\nu$, see Theorem \ref{thm: IE final}. In both cases, uniform and graded meshes, we use the following observation  $\Ba (g-\check g)(t)=\fInt^\alpha (g-\check g)'(t)$ because $(g-\check g)(0)=0$. 

The next lemma is needed to show the interpolation error estimate in Theorem \ref{thm: IE finalU} over a uniform mesh. 
  \begin{lemma} \label{lem: Mink} For $\gamma=1$ (the uniform time mesh), and for $0<\alpha,\,\nu<1,$ we have 
\[\sum_{j=3}^n \Big(\sum_{i=1}^{j-2} t_i^{\nu-2}(t_{j-1}-t_i)^{\alpha-1}\Big)^2
\le C\tau^{2\nu-4} \sum_{i=1}^{n-2} t_i^{2\alpha-2}\,. \]
\end{lemma}
\begin{proof} An application of the Cauchy-Schwarz inequality gives 
\[\Big(\sum_{i=1}^{j-2} i^{\nu-2}(j-i-1)^{\alpha-1}\Big)^2
\le  \sum_{i=1}^{j-2} i^{\nu-2} (j-i-1)^{2\alpha-2} \sum_{i=1}^{j-2} i^{\nu-2}
\le   C\sum_{i=1}^{j-2} (j-i-1)^{\nu-2} i^{2\alpha-2}\,.
\]
Now, summing over $j,$ and then, changing the order of summations, 
\[\sum_{j=3}^n \Big(\sum_{i=1}^{j-2} i^{\nu-2}(j-i-1)^{\alpha-1}\Big)^2
\le    C\sum_{i=1}^{n-2} i^{2\alpha-2} \sum_{j=i+2}^n  (j-i-1)^{\nu-2}\\
\le    C\sum_{i=1}^{n-2} i^{2\alpha-2}\,. \]
To complete the proof, we use the identities  $i=t_i /\tau$ and $j-1=t_{j-1}/\tau$. 
\end{proof}  

\begin{theorem}\label{thm: IE finalU} 
  Assume that the function $g$ satisfies the first regularity assumption in~\eqref{eq: g reg}.   Let $\check g$ be as defined in~\eqref{eq: g check}. Then, for $0<\alpha<1,$ $\gamma=1$, and for $1\le n\le N,$ we have  
  \[   
    \frac{1}{\tau}\sum_{j=1}^n  \Big|\int_{t_{j-1}}^{t_j} \Ba (g-\check g)(t)\,dt\Big|^2\le  
    C\tau^{2(\alpha+\nu)-1}+C\tau^{2\nu} \sum_{j=1}^{n-2}  \tau\,t_j^{2\alpha-2}\,.
  \]
\end{theorem}
    \begin{proof}
From the definition of $\check g$ in~\eqref{eq: g check}, we observe after some manipulations  that  
     \[(g-\check g)'(s) =\frac{1}{\tau} \int_{t_{i-1}}^{t_i}\int_q^s g''(z)\,dz\,dq,\]
and hence, by the first regularity assumption in \eqref{eq: g reg}, 
\[|(g-\check g)'(s)|\le \frac{1}{\tau}\int_0^{t_1} \int_q^s |g''(z)|\,dz\,dq
\le C\Big(s^{\nu-1}+\tau^{\nu-1}\Big)\le 
C\,s^{\nu-1},~~{\rm for}~~s\in I_1\,.\]
However, for $s \in I_j$ with   $j\ge 2$, we have 
$|(g-\check g)'(s)|\le
     \int_{t_{j-1}}^{t_j}|g''(q)|\,dq
       \le  c\tau t_j^{\nu-2}\,.$ Using these estimates,  
 \[|\fInt^\alpha (g-\check g)'(t)| \le  C
\int_0^{t}(t-s)^{\alpha-1} s^{\nu-1}\,ds\le C t^{\alpha+\nu-1},~~{\rm for}~~t\in I_1,\]
    while, for $t\in I_j$ with $j\ge 2,$
\begin{align*}
|\fInt^\alpha(g-\check g)'(t)| &\le  C\int_0^{t_1}(t-s)^{\alpha-1} s^{\nu-1}\,ds+C\tau\sum_{i=2}^j t_i^{\nu-2} \int_{t_{i-1}}^{\min\{t_i,t\}}
 (t-s)^{\alpha-1}\,ds\\
&\le C (t-t_1)^{\alpha-1} t_1^\nu
+C\tau^2\sum_{i=2}^{j-1} t_i^{\nu-2} 
 (t-t_i)^{\alpha-1}+C\tau t_j^{\nu-2} 
 (t-t_{j-1})^\alpha\\
&\le C\tau^2\sum_{i=1}^{j-2} t_i^{\nu-2} 
 (t-t_i)^{\alpha-1}+
 C\tau^2 t_j^{\nu-2} 
 (t-t_{j-1})^{\alpha-1}\,.
\end{align*}
From the above two achieved bounds, we have   
 \begin{align*}
 \Big(\int_{t_{j-1}}^{t_j}|\fInt^\alpha(g-\check g)'|\,dt\Big)^2 &  \le C\Big(\tau^3\sum_{i=1}^{j-2} t_i^{\nu-2}(t_{j-1}-t_i)^{\alpha-1}+
 \tau^{2+\alpha} t_j^{\nu-2} \Big)^2\\
 &  \le C\tau^6 \Big(\sum_{i=1}^{j-2} t_i^{\nu-2}(t_{j-1}-t_i)^{\alpha-1}\Big)^2+
 C\tau^{4+2\alpha} t_j^{2\nu-4},~ {\rm for}~j\ge 1\,.
\end{align*}
 Summing over $j$ and using Lemma \ref{lem: Mink} yield 
\[ \sum_{j=1}^n \Big|\int_{t_{j-1}}^{t_j} \fInt^\alpha (g-\check g)'(t)\,dt\Big|^2
  \le 
C\tau^{2\nu+2} \sum_{j=1}^{n-2} t_j^{2\alpha-2}+
 C\tau^{3+2\alpha} \sum_{j=1}^n \tau\,t_j^{2\nu-4}\,.\]
Finally, the desired result follows immediately after noting that the second summation  is bounded by $C t_1^{2\nu-3}=C\tau^{2\nu-3}\,.$ 
      \end{proof}
      
  We now turn to the interpolation errors for the case of a graded mesh where  the main result is in Theorem \ref{thm: IE final}.   Relying on  the Taylor series expansions with integral reminder, after  tedious calculations, we observe that   $(g-\check g)'(s) =e_1(s)+e_2(s)$ for $s \in I_i$, where for $i\ge 2$
\[
  e_1(s)=\int_s^{t_i}(q-s)g'''(q)\,dq-\frac{1}{2\tau_i}\int_{t_{i-1}}^{t_i}(q-t_{i-1})^2 g'''(q)\,dq, \quad
 e_2(s) = (s-t_{i-1/2})\, g''(t_i),
\]
while for $i=1$
\[
  e_1(s)=\frac{1}{t_1}\int_0^{t_1} \int_q^s g''(z)\,dz\,dq, \quad e_2(s)=0.
\]
Therefore, 
\begin{equation}\label{eq: IE deompose 0}
    \Big|\int_{t_{j-1}}^{t_j} \Ba (g-\check g)(t)\,dt\Big|\le \Big|\int_{t_{j-1}}^{t_j} \fInt^\alpha e_1(t)\,dt\Big| +\Big|\int_{t_{j-1}}^{t_j} \fInt^\alpha e_2(t)\,dt\Big|.
\end{equation}
The terms on the right-hand side will be estimated in the next two lemmas.
\begin{lemma} \label{lem: IE deompose 1} 
Assume that $\gamma > 2/\nu,$ then for $j\ge 1,$ we have 
\[
  \Big|\int_{t_{j-1}}^{t_j} \fInt^\alpha e_1(t)\,dt\Big| \le C\tau_j\tau^2 t_j^{\nu+\alpha-1-2/\gamma}.
\]
\end{lemma}
\begin{proof} Following the proof of (Lemma 3.2, Mustapha (2020)) with $\nu$ in place of $\sigma+\alpha/2$, we obtain the desired estimate. 
\end{proof}

\begin{lemma} \label{lem: IE deompose 2} 
Assume that $\gamma \ge 2/\nu$, then for $j\ge 1$, we have 
  \[
   \Big| \int_{t_{j-1}}^{t_j} \fInt^\alpha e_2(t)\,dt\Big|  \le  C\tau^2 \tau_j  t_j^{\nu+\alpha-1-2/\gamma}.
  \]
\end{lemma}
\begin{proof}
For $j=1,$ we have nothing to show because $e_2(t)=0.$ For $j\ge 2,$  following the proof of (Lemma 3.4, Mustapha (2020)) with $\nu$ in place of $\sigma+\alpha/2$ yields the desired bound. 
\end{proof}

By the decomposition in~\eqref{eq: IE deompose 0}, the estimates in Lemmas~\ref{lem: IE deompose 1} and~\ref{lem: IE deompose 2}, the use of the first  mesh property  in~\eqref{eq: mesh property 1},   and the Cauchy-Schwarz inequality, we observe that for $\gamma>2/\nu$,  
\[  \Big|\int_{t_{j-1}}^{t_j} \Ba (g-\check g)(t)\,dt\Big|^2\le C\Big(\tau^2 \int_{t_{j-1}}^{t_j} t^{\nu+\alpha-1-2/\gamma}\,dt\Big)^2\\
  \le C\tau^4 \tau_j \int_{t_{j-1}}^{t_j} t^{2(\nu+\alpha-1-2/\gamma)}\,dt.
\]
Dividing both side by $\tau_j$, then  summing over $j,$ we obtain the main interpolation error result (over a graded mesh)  in next theorem.  
\begin{theorem}\label{thm: IE final} 
  Assume that  $g$ satisfies the regularity assumption in~\eqref{eq: g reg}.   Let $\check g$ be as defined in~\eqref{eq: g check}. Then, for $0<\alpha<1,$  $\gamma>2/\nu$, and for $1\le n\le N,$ we have  
  \[   
    \sum_{j=1}^n \frac{1}{\tau_j} \Big|\int_{t_{j-1}}^{t_j} \Ba (g-\check g)(t)\,dt\Big|^2\le   C\tau^4\int_{t_1}^{t_n}   t^{2(\nu+\alpha-1-2/\gamma)}dt\le C\tau^4.
  \]
\end{theorem}

\section{Error analysis}\label{sec: error analysis}\label{sec: error}
In this section, we study the error bounds from the L1 time-stepping scheme~\eqref{fully}. 
Let us denote $\check u$ the piecewise linear polynomial interpolation of the solution $u$ of~\eqref{eq: ibvp} at the time nodes, that is,  
 $$\check u(t)=\frac{t_j-t}{\tau_j} u(t_{j-1})+\frac{t-t_{j-1}}{\tau_j} u(t_j),\quad{\rm  for}~~t_{j-1}\le t\le t_j,\quad{\rm with}~~1\le j\le N\,.$$ 
We introduce the following notations:
\[ 
	\Phi(t)=U(t)- \check u(t)\quad \mbox{and } \Psi(t)=u(t)- \check u(t).
\]
The next two lemmas form the foundation for estimating optimally $\|\Phi(t_n)\|$ (see the convergence results in Theorem~\ref{thm: time convergence}). To this end, we use the stability bound in Lemma~\ref{lem: stab 1} as well as the interpolation error estimate of Theorems \ref{thm: IE finalU} and \ref{thm: IE final}.   

\begin{lemma} \label{lem: error 1} 
For $1/2<\alpha < 1$ and for $1\le n\le N,$ the function $\Phi$ satisfies the following bound 
\[
 		\int_0^{t_n}\iprod{\Phi',\fInt^\alpha \Phi'}\,dt \le C\sum_{j=1}^n \frac{1}{\tau_j}\Big(\Big\|\int_{t_{j-1}}^{t_j} \partial_t^{1-\alpha} \Psi\,dt\Big\|_1^2 + \Big\|\int_{t_{j-1}}^{t_j} (\vF-\overline \vF)\partial_t^{1-\alpha} u \,dt\Big\|^2\Big).
\]
\end{lemma}
\begin{proof}
	Integrating \eqref{eq: ibvp} over $I_n$ and then subtracting  from  the time-stepping  numerical scheme in ~\eqref{fully}, and using the identity  $U-u=\Phi-\Psi,$ we get for $t\in I_n$ and $1\le n\le N$: 
	\begin{multline*}
		\Phi'(t)- \frac{1}{\tau_n}\int_{t_{n-1}}^{t_n} \vnabla \cdot\Big(\kappa_\alpha \partial_s^{1-\alpha}\vnabla \Phi- \overline \vF\partial_s^{1-\alpha} \Phi \Big)\,ds \cr 
		= - \frac{1}{\tau_n}\int_{t_{n-1}}^{t_n}\vnabla\cdot\Big(\kappa_\alpha \partial_s^{1-\alpha}\vnabla \Psi - \overline \vF\partial_s^{1-\alpha}  \Psi-(\vF-\overline \vF)\partial_s^{1-\alpha} u\Big)\,ds.
\end{multline*}
Note that the previous equation can be obtained by replacing $U$ with $\Phi$ and  $f$ with  $-\vnabla\cdot\Big(\kappa_\alpha \partial_s^{1-\alpha}\vnabla \Psi - \overline \vF\partial_s^{1-\alpha}  \Psi-(\vF-\overline \vF)\partial_s^{1-\alpha} u\Big)$ in~\eqref{fully}. 
Then, by adopting the proof of the stability result in  Lemma~\ref{lem: stab 1} and using the identity $\Phi(0)=0,$ we obtain 
\[ 	\int_0^{t_n}\iprod{\Phi',\fInt^\alpha \Phi'}\,dt \le   C\sum_{j=1}^n \frac{1}{\tau_j}\Big\|\int_{t_{j-1}}^{t_j} \Big[\kappa_\alpha \partial_t^{1-\alpha}\vnabla \Psi - \overline \vF\partial_t^{1-\alpha}  \Psi-(\vF-\overline \vF)\partial_t^{1-\alpha} u\Big] \,dt\Big\|^2.
\]
Therefore,  the desired result follows immediately.  
\end{proof}

\begin{lemma} \label{lem: error 3} Assume that the first regularity assumption in \eqref{eq: regularity2} is satisfied. For $0<\alpha < 1$ and for $\gamma \ge 1$, we have  
\[ 
	\sum_{j=1}^n \frac{1}{\tau_j} \Big\|\int_{t_{j-1}}^{t_j} \left[\vF-\overline \vF\right]\partial_t^{1-\alpha} u \,dt \Big\|^2 \le C \tau^{(1+2\alpha)\gamma}+C\tau^4\int_{t_1}^{t_n} t^{2\alpha-4/\gamma}\,dt,~~{\rm for}~~ 1\le n\le N\,.
\]
\end{lemma}
\begin{proof}
An elementary calculation shows that  
\[
	\vF(t)-\overline \vF(t) =  (t-t_{j-1/2})\vF'(t_{j-1})-\frac{1}{2}\int_{t_{j-1}}^{t_j} (t_j-s)\vF''(s)ds+\int_{t_{j-1}}^t (t-s)\vF''(s)ds,
\]
for $t \in I_j,$ where $t_{j-1/2}=(t_{j-1}+t_j)/2.$ Then, integration by parts and using the first regularity assumption in \eqref{eq: regularity2} give 
\begin{align*}
  	\int_{t_{j-1}}^{t_j} [\vF(t)-\overline \vF(t)]\partial_t^{1-\alpha} u(t) \,dt = \frac{ \vF'(t_{j-1})}{2}\int_{t_{j-1}}^{t_j} (t-t_{j-1})(t-t_j)\partial_t^{2-\alpha} u(t)\,dt\cr
 	+ \int_{t_{j-1}}^{t_j} \Big(\int_{t_{j-1}}^t (t-s)\vF''(s)ds -\frac{1}{2}\int_{t_{j-1}}^{t_j} (t_j-s)\vF''(s)ds\Big) \partial_t^{1-\alpha} u(t)\,dt,
 \end{align*}
and again, by the first regularity assumption in~\eqref{eq: regularity2} and the second mesh property in~\eqref{eq: mesh property 1}, it comes
\begin{align*}
 	\Big\|\int_{t_{j-1}}^{t_j} [\vF-\overline \vF]\partial_t^{1-\alpha} u \,dt\Big\|\le  C \,\tau_j^2 \int_{t_{j-1}}^{t_j} \Big(\|\partial_t^{2-\alpha} u\|+\|\partial_t^{1-\alpha} u\| \Big) \,dt\cr
 	\le  C   \int_{t_{j-1}}^{t_j} \tau_j^2 \Big( t^{\alpha-2} + t^{\alpha-1} \Big) \,dt
 	\le C\tau^2 \int_{t_{j-1}}^{t_j}  t^{\alpha-2/\gamma} \,dt,\quad \mbox{for } j\ge 2.
 \end{align*}
For $j=1$, similar arguments lead to  
\[
	\Big\|\int_0^{t_1} [\vF-\overline \vF]\partial_t^{1-\alpha} u \,dt \Big\| \le C \tau_1 \int_0^{t_1}  [t\|\partial_t^{2-\alpha} u(t)\|+\|\partial_t^{1-\alpha} u(t)\|] \,dt\le C \tau_1^{\alpha+1}.
\]
Gathering the above estimates completes the proof.
\end{proof}

 We are now ready to estimate the error of our $L1$ time-stepping scheme. An $O(\tau^{\sigma+\alpha/2})$-rate of convergence over a uniform time mesh (that is, $\gamma=1$) is proved,  where $\sigma$ is the regularity exponent occurring in~\eqref{eq: regularity2}. Furthermore, an optimal $O(\tau^2)$-rate of convergence  is achieved   for  the time mesh exponent $\gamma>2/(\sigma+\alpha/2)$. Thus, one should expect  $O(\tau^{\gamma(\sigma+\alpha/2)})$-rates for $1\le \gamma \le  2/(\sigma+\alpha/2))$.  
\begin{theorem}\label{thm: time convergence} 
	Let $1/2<\alpha < 1$, $U$ be the  solution of~\eqref{fully} and $u$ be the solution of problem~\eqref{eq: ibvp}. Assume that $u$  satisfies the regularity assumption in ~\eqref{eq: regularity2} for some  $0<\sigma\le \alpha$.  Then, for $1\le n\le N$,   
	\[
		\|U(t_n)-u(t_n)\| \le  C \times \begin{cases} \tau^{\sigma+\alpha/2},\quad &{\rm if}~~\gamma=1,\\
	\tau^2,\quad &{\rm if}~~\gamma>2/(\sigma+\alpha/2)\,.\end{cases} \]
\end{theorem}
\begin{proof}  An application of Lemma~\ref{lem: pointwise bound} yields $t_n^{\alpha-1}\|\Phi(t_n)\|^2 \le C \int_0^{t_n}\iprod{\Phi',\fInt^\alpha \Phi'}\,dt$. Combining this estimate with the estimate in Lemma~\ref{lem: error 1}, we deduce that
\[
	t_n^{\alpha-1} \|\Phi(t_n)\|^2 \le  C\sum_{j=1}^n \frac{1}{\tau_j}\Big(\Big\|\int_{t_{j-1}}^{t_j} \partial_t^{1-\alpha} \Psi\,dt\Big\|_1^2 +\Big\|\int_{t_{j-1}}^{t_j} (\vF-\overline \vF)\partial_t^{1-\alpha} u \,dt\Big\|^2\Big).
\]
 Applying  Theorems~\ref{thm: IE finalU}  and \ref{thm: IE final} with $u$ in place of $g$ and with $\nu=\sigma+\alpha/2$ owing to the regularity assumption in~\eqref{eq: regularity2}, and using that $1/2<\alpha < 1$, give
\[	\sum_{j=1}^n \frac{1}{\tau_j} \Big\|\int_{t_{j-1}}^{t_j} \partial_t^{1-\alpha} \Psi \,dt\Big\|_1^2 
	 \le Ct_n^{2\alpha-1}\times \begin{cases} \tau^{2\sigma+\alpha},&\quad\mbox{for } ~~\gamma=1,\\
	 \tau^4  t_n^{\alpha+2\sigma-4/\gamma},&\quad\mbox{for }~~\gamma > 2/(\sigma+\alpha/2).\end{cases}
\]
On the other hand, by  the achieved estimate in  Lemma~\ref{lem: error 3} and the inequalities $0<\sigma\le \alpha$ and $2\alpha-4/\gamma > -\alpha$ (this inequality follows from the assumptions  $\gamma> 2/(\sigma+\alpha/2)$ and $0<\sigma\le \alpha$), we have    
\[ 
	\sum_{j=1}^n \frac{1}{\tau_j} \Big\|\int_{t_{j-1}}^{t_j} \left[\vF-\overline \vF\right]\partial_t^{1-\alpha} u \,dt \Big\|^2
	 \le C\times \begin{cases}  \tau^{1+2\alpha}\le \tau^{2\sigma+\alpha},&~\mbox{for } ~\gamma=1,\\
 \tau^4+\tau^4 t_n^{2\alpha-4/\gamma+1},&~\mbox{for }~\gamma > 2/(\sigma+\alpha/2).\end{cases}
\]
 Gathering the above estimates  completes the proof.
\end{proof}

\section{A fully discrete scheme and numerical  results}\label{sec: numerical results}
In this section, we introduce a fully discrete scheme for problem~\eqref{eq: ibvp}. We use this scheme to illustrate numerically the convergence of the proposed $L1$ time-stepping scheme  and compare them with the  theoretical convergence  proven in Theorem~\ref{thm: time convergence}.

\subsection{A fully discrete scheme}\label{sec: fully22}
 
We discretize the L1 time-stepping scheme~\eqref{fully} in space using the standard Galerkin continuous piecewise linear finite element method (P1 finite elements)  to obtain a fully-discrete solution $U_h$. 
To this end, we introduce a family of regular (conforming) triangulation $\mathcal{T}_h$ of the domain $\overline{\Omega}$ and denote $h=\max_{K\in \mathcal{T}_h}(h_K)$, where $h_{K}$ denotes the diameter of the element $K$. 
Let $V_h \subset H^1_0(\Omega)$ denote the usual space of continuous, piecewise-linear functions on  $\mathcal{T}_h$ that vanish on $\partial \Omega$. Let $\W(V_h)\subset  C([0,T];V_h)$ denote the space of linear polynomials on $[t_{n-1},t_n]$ for $1\le  n\le N$,  with coefficients in $V_h.$

Taking the inner product of~\eqref{fully} with a test function $v \in H^1_0(\Omega)$, and applying the first Green identity, the semi-discrete $L1$ solution $U$ satisfies for $n\ge 1$ ($U(0)=u_0$)
\begin{gather} \label{fully0}
 \iprod{U(t_n)-U(t_{n-1}),v}+ \int_{t_{n-1}}^{t_n}\iprod{[\kappa_\alpha \partial_t^{1-\alpha}\vnabla U- \overline \vF\partial_t^{1-\alpha} U],\vnabla v} \,dt
= \int_{t_{n-1}}^{t_n} \iprod{f,v}\,dt.
 \end{gather}
 
Motivated by~\eqref{fully0}, we define our fully-discrete  solution $U_h\in \W(V_h)$ as: with $U_h^n:=U_h(t_n)$, 
\begin{equation} \label{fully 2 2}
\iprod{U_h^n-U_h^{n-1},v_h}+ \int_{t_{n-1}}^{t_n}\iprod{[\kappa_\alpha \partial_t^{1-\alpha}\vnabla U_h- \overline \vF\partial_t^{1-\alpha} U_h](t),\vnabla v_h}\,dt
= { \int_{t_{n-1}}^{t_n}\iprod{f(t),v_h}\,dt}
\end{equation}
$\forall\, v_h\in V_h$, for $1\le n\le N$. The discrete initial solution $U_h^0\in V_h$ approximates of the initial data $u_0$. One can choose $U_h^0= R_h u_0$,  where $R_h:H^1_0(\Omega) \mapsto V_h$ is the Ritz projection defined by 
$\iprod{\kappa_\alpha \vnabla (R_h w-w),\vnabla v_h}= 0$  for all $v_h \in V_h.$ For non-smooth $u_0$, $U_h^0$ can be defined \textit{via} a simpler $L_2$ projection of $u_0$ over the space $V_h$.  
 
To show the stability of the fully-discrete scheme, we follow line-by-line the proof of Theorem~\ref{thm: stability} and conclude that the stability result of the theorem remains valid for $u_{0h}=U_h^0$ in place of $u_0.$ 
The uniqueness of the fully discrete solution follows immediately from this result. Further,
since the numerical scheme~\eqref{fully 2 2} reduces to a finite square linear system at each time level $t_n$, see~\eqref{eq: matrix form} below, the existence of $U_h^n$ follows from its uniqueness. 

Concerning the error analysis of the fully-discrete scheme, an additional error term of order $O(h^2)$ is expected to appear under certain regularity assumptions on $u$. As mentioned in the introduction, the convergence analysis of spatial discretization \textit{via} Galerkin finite elements was recently studied by few authors, see (Huang et al. (2020), Le et al. (2016) and (2018), McLean and Mustapha (2021)).

To write the fully-discrete scheme~\eqref{fully 2 2} in a matrix form, let $d_h:=\dim V_h$. Then, for $1\le p\le d_h$, let $\phi_p\in V_h$ denote the $p$-th basis function associated with the $p$-th interior node $x_p$, so that $\phi_p(x_q)=\delta_{pq}$ and $U^n_h({x})=\sum_{p=1}^{d_h}U^n_{h,p} \phi_p(x)$. We define $d_h\times d_h$ matrices:  $\matG=\left[\iprod{\kappa_\alpha \vnabla \phi_{q},\vnabla \phi_p}-\iprod{\vF^{n-\frac{1}{2}}  \phi_{q},\vnabla \phi_p}\right],$ $\matM=[\iprod{\phi_q,\phi_p}],$ and the $d_h$-dimensional column vectors ${\bf U}_h^n$  and ${\bf f}^n$ with components~$U^n_{h,p}$ and $\int_{t_{n-1}}^{t_n} \iprod{f(t),\phi_p}\,dt,$ respectively.   Thus, the  scheme~\eqref{fully 2 2} has the following matrix representations: 
\begin{multline*}
  \matM({\bf U}_h^n-{\bf U}_h^{n-1})+ \frac{\tau_n^\alpha}{\Gamma(\alpha+2)}\matG({\bf U}_h^n+\alpha {\bf U}_h^{n-1})
	={\bf f}^n+\omega_{\alpha+1}(\tau_n) \matG {\bf U}_h^{n-1} \\
	-\big(\omega_{\alpha+1}(t_n)-\omega_{\alpha+1}(t_{n-1})\big) \matG {\bf U}_h^0
	 -\matG\sum_{j=1}^{n-1}\frac{\omega_{n,j}^\alpha}{\tau_j}({\bf U}_h^j-{\bf U}_h^{j-1}),~~{\rm for}~~1\le n\le N,
	\end{multline*}
 with
$  \omega_{n,j}^\alpha =\big(\omega_{\alpha+2}(t_n-t_{j-1})-\omega_{\alpha+2}(t_{n-1}-t_{j-1})\big)-\big(\omega_{\alpha+2}(t_n-t_j)-\omega_{\alpha+2}(t_{n-1}-t_j)\big).$
Thus, with $\matS=\matM+ \frac{\tau_n^\alpha}{\Gamma(\alpha+2)}\matG,$  and ${\bf W}_h^n= {\bf U}_h^n-{\bf U}_h^{n-1}$, \eqref{fully 2 2} has the following compact matrix representations: 
\begin{equation}\label{eq: matrix form}
\matS {\bf W}_h^n
	= {\bf f}^n-\big(\omega_{\alpha+1}(t_n)-\omega_{\alpha+1}(t_{n-1})\big) \matG {\bf U}_h^0-\matG\sum_{j=1}^{n-1}
	\frac{\omega_{n,j}^\alpha}{\tau_j}  {\bf W}_h^j,~~{\rm for}~~1\le n\le N\,.
\end{equation}

For one-dimensional problems and the P1 finite element method considered here, the matrices $\matS$ (as well as $\matG$ and $\matS$) have tri-diagonal structures. The resolution of~\eqref{eq: matrix form} then raises no particular issue since, as pointed out before, $\matS$ is non-singular.
In higher-dimension, the matrices will retain the sparse character of the mass $\matM$ and $\matG$ matrices resulting from P1 discretization, enabling the use of efficient direct or iterative solvers.
Note that in the case of zero driving force, $\vF=\mvec{0}$, the matrix $\matS$ is symmetric and positive definite (SPD) such that methods for SPD systems can be employed, \textit{e.g.}, Cholesky decomposition or conjugate gradient methods). In contrast with the excellent features of $\matS$, the system of equations~\eqref{eq: matrix form} requires storing the solution ${\bf U}^j_h$ (or increments ${\bf W}^n_h$) at all previous time steps to evaluate the right-hand-side. This requirement constitutes a significant limitation for problems with $d>1$ and large spatial mesh, an issue that is not specific to our method.

\subsection{Numerical results}\label{sec: numres}

In our  test problem, we consider the fractional model problem~\eqref{eq: ibvp} in one-dimension ($d=1$), with $\Omega=(0,1)$, $F(x,t) =  \sin(t)-x$,  $T = 1$,  and $\kappa_\alpha = 1$. The performance of the considered spatial Galerkin finite element scheme was previously demonstrated theoretically and numerically for both smooth and non-smooth data in (Le et al. (2016) and  (2018), McLean and Mustapha (2021)). In the following, we focus on the numerical illustration of the performance of the $L1$ time-stepping discretization on two typical examples with infinite Fourier spectra, using a sufficiently refined uniform spatial mesh  of size $h=1/2000$ to make the time error dominant. 

{\bf Example 1.} We set the initial data $u_0(x)=x(1-x)$ and choose $f$ so that the exact solution is 
\begin{equation}\label{eq: u series}
    u(x,t)=8\sum_{m=0}^\infty \lambda_m^{-3}\sin( \lambda_m x) E_{\alpha,1}(-\lambda_m^2 t^{\alpha}),\quad \lambda_m:=(2m+1)\pi,
\end{equation}
where $E_{\mu,\beta}(t):=\sum_{p=0}^\infty\frac{t^p}{\Gamma(\mu p+\beta)}$ is the Mittag-Leffler function, with parameters $\mu,\beta >0.$ The source term in that case is 
\[
  f(x,t)=8t^{\alpha-1} \sum_{m=0}^\infty \lambda_m^{-3}
  \Big[\lambda_m\cos( \lambda_m x)(\sin(t) -x)-\sin( \lambda_m x)\Big] E_{\alpha,\alpha}(-\lambda_m^2 t^{\alpha}).
\]

In the following, we measure the error in the numerical solution by computing $\epsilon_N=\max_{1\le n\le N} \|U_h^n-u(t_n)\|,$
where $N$ is the number of time subintervals. The spatial $L_2$-norm is evaluated using the two-point Gauss quadrature rule per element. 
Figure~\ref{fig2} shows the evolution in time of the pointwsie error $\|U_h^n-u(t_n)\|$ for different time-graded mesh.
From this plot, one can appreciate the global decay with the time mesh exponent $\gamma$ of the pointwise errors and the higher error in the neighborhood of $t\approx 0$.
\begin{figure}
\begin{center}
\includegraphics[width=9cm, height=5cm]{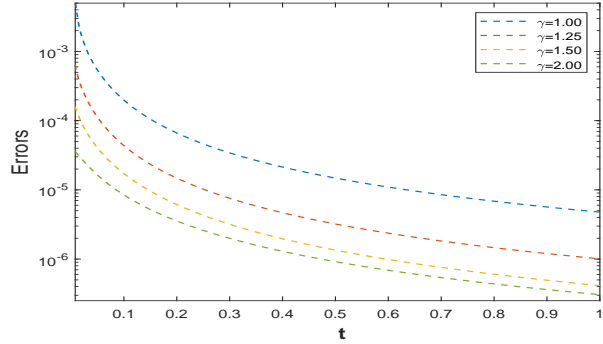}
\caption{Pointwise error $\|U_h^n-u(t_n)\|$ as a function of $t_n$ for $N=128$, $\alpha=0.6$, and different values of  $\gamma\ge 1$.}
\label{fig2}
\end{center}
\end{figure}

The time convergence rate $r_t$ is subsequently calculated from the relation $r_t \approx {\rm log}_2\big(\epsilon_N/\epsilon_{2N}\big).$ It is not difficult to show that the data of the problem satisfy the conditions of the stability and convergence Theorems~\ref{thm: stability} and~\ref{thm: time convergence} for $1/2<\alpha < 1$. 
Furthermore, because of the inequality 
\[
  \Big|\frac{d^j}{dt^j}  E_{\alpha,1}(-\lambda_m^2 t^{\alpha})\Big|\le C(\lambda_m^2 t^{\alpha})^{-\mu} t^{-j}, \quad \mbox{for } j \in \{1,2,3,\cdots\} \mbox{ and }|\mu|\le 1,
\]
the first regularity estimate in~\eqref{eq: regularity2} holds true for $0<\alpha<1$, and the second one is valid for $\sigma=\alpha^-/4$. This result is supported by the regularity analysis presented in (McLean (2010) and McLean et al. (2020)), because $u_0 \in \dot H^{2.5^-}(\Omega)\cap H^1_0(\Omega)$ and the source term $f$ satisfies the inequality $t^j \|f^{(j)}(t)\|_1\le C t^{\frac{3\alpha}{4} -1}$ for $t>0,$ with $j=0,1,2,3$. 
 Therefore, for $1/2<\alpha < 1$, by  Theorem~\ref{thm: time convergence}, 
\[r_t=\begin{cases} \sigma+\alpha/2\approx  3/(4\alpha^-)&\quad{\rm  when}~~ \gamma=1,\\
   2&\quad{\rm when}~~ \gamma> 2/(\sigma+\alpha/2) \approx  8/(3\alpha^-)\,.\end{cases}\] 

 Table~\ref{table 1 FFP} presents $\epsilon_N$ and  $r_t$, for $\alpha=0.7$ and for different   choices of   $\gamma \ge 1$. As expected, $\epsilon_N$ and $r_t$ improve with $\gamma$. The empirical convergence rate is found to be better than expected with $r_t\approx (\sigma+\alpha)\gamma$ for $1 \le \gamma \le  2/(\sigma +\alpha)\approx 2.3.$ Further,  Table~\ref{table 2 FFP} reports similar  convergence rates for  $\alpha =0.5$ (top part) and $\alpha=0.4$ (bottom part), that is,  when $\alpha$ is outside the range covered by the theory.

\begin{table}[hbt]
\begin{center}
\begin{tabular}{|c|cc|cc|cc|}
\hline
$N$&\multicolumn{2}{c|}{$\gamma=1$ }
&\multicolumn{2}{c|}{$\gamma=1.6$}
&\multicolumn{2}{c|}{$\gamma=2.3$}\\
\hline
    16& 2.039e-02&         &   3.787e-03&         &  8.847e-04&         \\
    32& 1.022e-02&   0.9964&   1.417e-03&   1.4187&  2.234e-04&   1.9856\\
    64& 4.976e-03&   0.1038&   5.446e-04&   1.3792&  5.638e-05&   1.9863\\
   128& 2.578e-03&   0.9486&   2.065e-04&   1.3990&  1.421e-05&   1.9880\\
   256& 1.417e-03&   0.8640&   7.816e-05&   1.4018&  3.588e-06&   1.9863\\
    \hline
\end{tabular}
\caption{Errors $\epsilon_N$ and convergence rates $r_t\approx (\sigma+\alpha)\gamma$ for $\alpha=0.7$, and different choices of $1\le \gamma\le 2/(\sigma+\alpha)$.}
\label{table 1 FFP}
\end{center}
\end{table}

\begin{table}[hbt]
\begin{center}
\begin{tabular}{|c|c|c|c|c|}
\hline
\multicolumn{5}{|c|}{$\alpha=0.5$}\\
\hline
$N$ & $\gamma=1$ & $\gamma=1.6$ &$\gamma=2.4$ & $\gamma=3.2$\\
\hline
    32& 2.5e-02~~~         0.458& 6.7e-03~~~   1.082&1.0e-03~~~   1.568&   2.6e-04~~~   1.953\\
    64& 1.7e-02~~~         0.551& 3.0e-03~~~   1.139&3.7e-04~~~   1.489&   6.8e-05~~~   1.954\\
   128& 1.1e-02~~~         0.634& 1.5e-03~~~   1.065&1.3e-04~~~   1.499&   1.7e-05~~~   1.962\\
   256& 6.7e-03~~~         0.693& 7.3e-04~~~   0.999&4.6e-05~~~   1.502&   4.5e-06~~~   1.966\\    
\hline
\hline
\multicolumn{5}{|c|}{$\alpha=0.4$ }\\
\hline
$N$&$\gamma=1$&$\gamma=2$ &$\gamma=3$&$\gamma=4$\\
\hline
 32&   2.9e-02~~~   0.276&   6.4e-03~~~   1.029&   9.7e-04~~~   1.604&   2.8e-04~~~   1.913\\
 64&   2.3e-02~~~   0.331&   2.9e-03~~~   1.126&   3.4e-04~~~   1.496&   7.4e-05~~~   1.921\\
128&   1.8e-02~~~   0.388&   1.4e-03~~~   1.089&   1.2e-04~~~   1.497&   1.9e-05~~~   1.936\\
256&   1.3e-02~~~   0.444&   6.8e-04~~~   1.016&   4.3e-05~~~   1.502&   5.0e-06~~~   1.945\\
    \hline
\end{tabular}
\caption{Errors $\epsilon_N$ and convergence rates $r_t$ for different choices of $1\le \gamma\le 2/(\sigma+\alpha)$. It is observed that $r_t\approx {(\sigma+\alpha)\gamma}$. }
\label{table 2 FFP}
\end{center}
\end{table}

{\bf Example 2.}  The second example considers a solution with lower regularity. We choose  $u_0(x)=x$ on $[0,1/2]$ and $1-x$ on $[1/2,1]$. Thus, $u_0 \in H^{1.5^-}(\Omega)$. The source term $f$ is chosen so that 
 \begin{equation}\label{eq: u series ex2}
u(x,t)=4\sum_{m=0}^\infty (-1)^m\lambda_m^{-2}\sin( \lambda_m x) E_{\alpha,1}(-\lambda_m^2 t^{\alpha}),\quad \lambda_m:=(2m+1)\pi.
\end{equation}
Following similar arguments as of the previous example, we deduce that the first regularity estimate in~\eqref{eq: regularity2} holds for $0<\alpha<1$, while the second one is valid only for $\sigma=-\alpha^+/4<0.$ Thus, the required regularity assumptions of Theorems~\ref{thm: time convergence} are not satisfied. To avoid dealing with negative values of $\sigma,$ we focus on the weighted error $t_n^{\alpha/4}\|U_h^n-u(t_n)\|$ (so $\sigma^*=\sigma+\alpha/4 \approx 0$) and the corresponding convergence rates. We denote  $\epsilon^*_N = \max_{1\le n\le N} \, t_n^{\alpha/4}\|U^n_h-u(t_n)\|,$
and $r_t^*$ the associated rate of convergence.

Table~\ref{table 2 FFP ex2} reports $\epsilon^*_N$  and $r_t^*$ for $\alpha=0.4$, 0.6 and 0.8, and different $\gamma$ in the range $[1,2/\alpha]$. As in the previous example, the empirical results confirmed  $r_t^* \approx \min\{\gamma(\sigma^*+\alpha),2\}\approx \min\{\gamma\alpha,2\}$-rates  of convergence. 
Therefore, at time level $t_n,$ we may conclude that the error $\|U_h^n-u(t_n)\|\le Ct_n^{-\alpha/4}  \tau^{\min\{\gamma \alpha,2\}}$. 
When $\gamma=1$ (uniform time meshes), such as estimate is expected in the limiting case $\alpha \to 1$ (the fractional Fokker-Planck~\eqref{eq: ibvp} reduces to the classical Fokker-Planck  model).

\begin{table}[hbt]
\begin{center}
\begin{tabular}{|c|cc|cc|cc|cc|}
\hline
\multicolumn{5}{|c|}{$\alpha=0.4$ }\\
\hline
$N$&{$\gamma=1$ }
&{$\gamma=2$}
&{$\gamma=3$}&{$\gamma=5$}\\
\hline
   32& 3.4e-02~~~   0.363&   6.75e-03~~~   9.716&  1.7e-03~~~   1.159& 1.8e-04~~~ 1.928\\
   64& 2.5e-02~~~   0.409&   3.72e-03~~~   8.581&  7.3e-04~~~   1.186& 4.7e-05~~~ 1.937\\
  128& 1.8e-02~~~   0.452&   2.18e-03~~~   7.730&  3.2e-04~~~   1.198& 1.2e-05~~~ 1.949\\
  256& 1.3e-02~~~   0.482&   1.27e-03~~~   7.782&  1.4e-04~~~   1.199& 3.1e-06~~~ 1.958\\
    \hline
\multicolumn{5}{|c|}{$\alpha=0.6$ }\\
\hline
$N$&{$\gamma=1$ }
&{$\gamma=2$}
&{$\gamma=2.6$}&{$\gamma=3.3$}\\
\hline
   32& 1.9e-02~~~   0.748&   2.1e-03~~~   1.155&   6.1e-04~~~   1.557&   1.7e-04~~~   1.965\\
   64& 1.1e-02~~~   0.741&   9.3e-04~~~   1.196&   2.1e-04~~~   1.559&   4.3e-05~~~   1.982\\
  128& 7.1e-03~~~   0.654&   4.1e-04~~~   1.199&   7.2e-05~~~   1.559&   1.1e-05~~~   1.986\\
  256& 4.7e-03~~~   0.579&   1.8e-04~~~   1.199&   2.4e-05~~~   1.561&   2.8e-06~~~   1.987\\
  \hline
\multicolumn{5}{|c|}{$\alpha=0.8$ }\\
\hline
$N$&{$\gamma=1$ }
&{$\gamma=1.5$}
&{$\gamma=2$}&{$\gamma=2.5$}\\
\hline
   32&  0.0e-03~~~   0.824&   2.4e-03~~~   1.181&   5.9e-04~~~   1.599&   1.7e-04~~~   2.031\\
   64&  5.4e-03~~~   0.733&   1.0e-03~~~   1.200&   1.9e-04~~~   1.599&   4.2e-05~~~   2.000\\
  128&  3.1e-03~~~   0.780&   4.5e-04~~~   1.199&   6.5e-05~~~   1.599&   1.0e-05~~~   1.997\\
  256&  1.8e-03~~~   0.801&   2.0e-04~~~   1.199&   2.1e-05~~~   1.601&   2.7e-06~~~   1.969\\  
    \hline
  \end{tabular}
\caption{Weighted errors $\epsilon^*_N$ and associated convergence rates for Example 2, with  different choices of 
$\alpha$ and $\gamma\ge1$. The time convergence rate is $\approx {(\sigma^*+\alpha)\gamma}$ for $1\le \gamma\le 2/(\sigma^*+\alpha)$, where in this example $\sigma^*=\sigma+\alpha/4 \approx 0$. }
\label{table 2 FFP ex2}
\end{center}
\end{table}

\end{document}